\documentclass[a4paper,10pt]{article}
\usepackage[utf8]{inputenc}
\usepackage{amsmath,amsthm,amssymb,verbatim}
\input{xy}
\xyoption{all}
\newcommand{\C}{{\mathbf C}}		
\renewcommand{\P}{{\mathbf P}}		
\renewcommand{\d}{{\mbox d}}		
\newcommand{\gl}{\mathfrak{gl}}		
\newcommand{\h}{{\mathfrak h}}		
\renewcommand{\u}{{\mathfrak u}}		
\newcommand{\G}{{\mathcal G}}		
\newcommand{\Ct}{\widehat{\C}}		
\newcommand{\Et}{\widehat{E}}		
\renewcommand{\O}{{\mathcal O}}		
\newcommand{\E}{{\mathcal E}}		
\newcommand{\End}{{\mathcal End}}	
\newcommand{\Hom}{{\mathcal Hom}}	
\newcommand{\Ext}{{\mathcal Ext}}	
\newcommand{\M}{{\mathcal M}}		
\renewcommand{\L}{{\mathbf L}}		
\newcommand{\Nahm}{{\mathcal N}}	
\newcommand{\Mt}{\widehat{\M}}		
\newcommand{\e}{{\mathbf e}}		
\newcommand{\m}{{\mathbf m}}		
\renewcommand{\H}{{\mathbf H}}		

\DeclareMathOperator{\tr}{tr}			
\DeclareMathOperator{\diag}{diag}		
\DeclareMathOperator{\Gl}{Gl}			
\DeclareMathOperator{\ad}{ad}		
\DeclareMathOperator{\red}{red}		
\DeclareMathOperator{\Pic}{Pic}		
\DeclareMathOperator{\GlobExt}{Ext}	
\DeclareMathOperator{\Mukai}{Mukai}
\DeclareMathOperator{\res}{res}		
\DeclareMathOperator{\gr}{gr}		

\newtheorem{lem}{Lemma}[section]
\newtheorem{rem}{Remark}[section]
\newtheorem{prop}{Proposition}[section]
\newtheorem{thm}{Theorem}

\title{The Plancherel theorem for Fourier--Laplace--Nahm transform for connections on the projective line}
\author{Szil\'ard Szab\'o}
\date{\today}

\begin{document}

\maketitle

\begin{abstract}
We prove that the Fourier--Laplace--Nahm transform for connections on the projective line is a hyper-K\"ahler isometry.
\end{abstract}

\section{Introduction}

The aim of this paper is to prove the 
\begin{thm}\label{thm:main}
The Fourier--Laplace--Nahm transform for solution of the Hitchin--equations on the complex projective line 
with logarithmic singularities at finitely many points and a second-order pole at infinity 
is a  hyper-K\"ahler isometry between the moduli spaces. 
\end{thm}
Fourier--Laplace--Nahm transform for such solutions was defined and studied in \cite{Sz-these} and subsequently 
\cite{Sz-FourierLaplace}, \cite{ASz}; the definition of the moduli spaces involved and their hyper-K\"ahler structure was studied in \cite{Biq-Boa}. 

As the Fourier--Laplace--Nahm transform is a non-linear analog of classical Fourier transform and the 
Riemannian metric on the moduli spaces is given by the $L^2$-pairing of harmonic representatives 
of tangent vectors, the above theorem is a non-linear analog of the classical Plancherel theorem 
stating that Fourier transform is an isometry on the space of $L^2$ functions on the line. 
The isometry property of similar transforms has been studied by several authors \cite{bvb}, \cite{mac}. 
They use various methods: the proof in \cite{bvb} is heavily computational and makes use of  
Clifford algebra identities to link the Green's operator of the original  solution to the Laplacian 
of the transformed one; \cite{mac} directly links the hyper-K\"ahler potentials of the spaces of instantons 
to that of ADHM--data. This latter method is unlikely to work in the case of solutions of the Hitchin--equations  
because no explicit potential is known to exist for the hyper-K\"ahler metric in this case. 
However, it is suggested as a concluding remark in \cite{bvb} that their computations should have a purely 
holomorphic counterpart in terms of one of the complex structures  --- this is the point of view we adopt here, by making explicit the 
relationship of the holomorphic symplectic structure for the Dolbeault complex structure 
on the moduli space of solutions to the Hitchin--equations to a natural holomorphic symplectic structure 
(the Mukai--pairing) defined on moduli spaces of some sheaves on surfaces. 
The sheaf to take for a Higgs--bundle will be its spetral sheaf, defined on a ruled surface. 

Fourier--Laplace transformation of holonomic $\mathcal{D}$-modules has been studied by many authors, e.g. \cite{Mal} \cite{Sab}. 
From the point of view of Hamiltonian systems this transform is also closely related to Harnad's duality \cite{Har}. 
In particular, as explained in \cite{Boa-irreg}  in the Fuchsian case this transformation corresponds to the reflection about the central root 
in the Kac--Moody root system associated to a star-shaped quiver; more generally, in the case of second-order 
singularities treated in this paper it correponds to switching between two of the three possible readings of the 
associated complete bipartite graph. 

In Section \ref{sec:hK}, we set the notation and collect some known facts about hyper-K\"ahler manifolds. 
In Section \ref{sec:mod-Hit} we describe the hyper-K\"ahler moduli spaces of singular solutions of the Hitchin--equations 
along the lines of \cite{hit} and \cite{Biq-Boa}. 
In Section \ref{sec:sheaves} we recall the Beauville--Narasimhan--Ramanan construction and discuss about 
moduli spaces of sheaves on Poisson manifolds with emphasis on the Mukai pairing. 
In Section \ref{sec:FLN} we define the Fourier--Laplace--Nahm transform on the moduli spaces 
and prove Theorem \ref{thm:main}. 

The author would like to thank Philip Boalch for constant encouragement and Jacques Hurtubise and Tony Pantev 
for useful discussions. The paper was written under support of the Advanced Grant 
”Arithmetic and physics of Higgs moduli spaces” no. 320593 of the European Research Council and the Lend\"ulet 
"Low Dimensional Topology" program of the Hungarian Academy of Sciences. 


\section{Hyper-K\"ahler geometry}\label{sec:hK}

Recall that a hyper-K\"ahler manifold is a Riemannian manifold $(M,g)$ of real dimension $4k$ 
together with integrable almost-complex structures $(I,J,K)$ on its tangent bundle, 
satisfying 
\begin{itemize}
\item the quaternion relations $I^2=J^2=K^2=-\mbox{Id}_{TM}$, $IJ=-JI=K$
\item the polarisation conditions: $\omega_I(\cdot , \cdot) = g(\cdot ,I \cdot)$ 
is a symplectic form, i.e. non-degenerate and closed; and similarly for 
$\omega_J(\cdot , \cdot) = g(\cdot ,J \cdot)$ and $\omega_K(\cdot , \cdot) = g(\cdot ,K \cdot)$. 
\end{itemize}
A holomorphic symplectic manifold is a complex analytic manifold $X$ of complex 
dimension $2k$ endowed with a holomorphic non-degenerate $2$-form $\Omega\in H^{2,0}(X,\C)$. 
Any hyper-K\"ahler manifold is a holomorphic symplectic manifold in three different ways: 
\begin{itemize}
\item $X_I = (M,I)$, $\Omega_I = \omega_J + \sqrt{-1}\omega_K$
\item $X_J = (M,J)$, $\Omega_J = \omega_K + \sqrt{-1}\omega_I$
\item $X_K = (M,K)$, $\Omega_K = \omega_I + \sqrt{-1}\omega_J$.
\end{itemize}
If $(M, g, I, J, K)$ and $(\hat{M},\hat{g},\hat{I},\hat{J},\hat{K})$ are hyper-K\"ahler manifolds of the same dimension $4k$ 
and $f:M\to \hat{M}$ a diffeomorphism then $f$ is said to be a hyper-K\"ahler isometry 
if it maps $g$ to $\hat{g}$, $I$ to $\hat{I}$, $J$ to $\hat{J}$ and $K$ to $\hat{K}$. 
In concrete terms, if the tangent map of $f$ at $m\in M$ is denoted by $D_mf$ then these requirements read as follows: 
for any $m\in M$ we have $g_m = \hat{g}_{f(m)} \circ (D_mf \otimes D_mf)$, $D_mf \circ I_m= \hat{I}_{f(m)} \circ D_mf$,  
$D_mf \circ J_m= \hat{J}_{f(m)} \circ D_mf$ and $D_mf \circ K_m= \hat{K}_{f(m)} \circ D_mf$. 
It is clear that this implies that the map 
$$
	f: X_I \to \hat{X}_{\hat{I}}
$$
is a complex analytic map and $f^*\Omega_{\hat{I}} = \Omega_I$, and similarly for $J$ and $K$. 
Conversely, one might ask whether a map satisfying the above conditions for the complex 
structure $I$ is necessarily a hyper-K\"ahler isometry. This is not quite true, however 
we have the following well-known statement. 

\begin{prop}\label{prop:holomsymp-hypK}
If $f:M\to \hat{M}$ maps $I,J$ to $\hat{I}, \hat{J}$ respectively and $f^*\Omega_{\hat{I}} = \Omega_I$ then $f$ is 
a hyper-K\"ahler isometry. 
\end{prop}

\begin{proof}
As $K=IJ$ and $\hat{K}=\hat{I} \hat{J}$ it follows from the assumptions on $I,J,\hat{I}, \hat{J}$ that $f$ maps $K$ to $\hat{K}$. 
On the other hand, from the assumption on $ \Omega_I, \Omega_{\hat{I}}$ it follows that $\omega_J = \Im \Omega_I$ 
is pulled back to $\omega_{\hat{J}} = \Im \Omega_{\hat{I}}$ by $f$. We infer that for any pair of vectors $Y,Z\in T_mM$ 
we have 
$$
	\hat{g}(D_mf(Y),D_mf(Z)) = \omega_{\hat{J}}(D_mf(Y),-\hat{J}_{f(m)}D_mf(Z)) = \omega_J(Y,,-J_mZ) = g(Y,Z), 
$$
i.e. $g$ is mapped to $\hat{g}$ by $f$.
\end{proof}

\section{Moduli spaces of singular solutions of the Hitchin--equations on curves}\label{sec:mod-Hit}

Let $C$ be a projective curve over $\C$, $E$ a smooth Hermitian vector bundle of rank $r$ on $C$, 
$\nabla_A$ a unitary connection in $E$ with curvature $F_A$ and $\Phi$ a smooth section of the vector bundle 
$\Omega^{1,0}\otimes \gl(E)$ over $C$ (where $\gl(E)$ refers to the bundle of complex endomorphisms of $E$). 
We denote by $\Phi^*$ the $(0,1)$-form conjugate to $\Phi$ with respect to the Hermitian structure: 
if in a local holomorphic coordinate $z$ of $C$ and unitary frame we have $\Phi = \phi \d z$ for 
a smooth $1$-parameter family $\phi(z)$ with values in $\gl_r(\C)$, then $\Phi^*=\bar{\phi}\d \bar{z}$.  
Further, we let $\bar{\partial}_A$ stand for the $(0,1)$-part of $\nabla_A$. Hitchin's equations \cite{hit} then read 
\begin{align}
	F_A + [\Phi,\Phi^*] & = 0 \label{hit1} \\
	\bar{\partial}_A \Phi & = 0. \label{hit2}
\end{align}
The solutions to this equation form a (finite-dimensional) complete hyper-K\"ahler manifold with the Atiyah--Bott metric coming 
from an infinite-dimensional quotient construction; let us briefly recall the description of this hyper-K\"ahler structure, 
or rather the holomorphic symplectic structure associated to one of the complex structures $I$ (the Dolbeault complex structure). 
Here we identify the space of unitary connections $A$ with the space of holomorphic structures $\bar{\partial}_A$ 
via the Chern--connection construction: for any unitary connection we associate its $(0,1)$-part $\bar{\partial}_A$ as above; 
conversely for any holomorphic structure $\bar{\partial}$ in $E$ there exists a unique unitary connection $A$ (known 
as the Chern--connection) for which $\bar{\partial} = \bar{\partial}_A$. 
Let us define $\G$ (respectively $\G_{\C}$) as the group of $\u(E)$- (respectively $\gl(E)$-)valued gauge transformations of $E$. 
Then equation (\ref{hit2}) is $\G_{\C}$-invariant and in every $\G_{\C}$-orbit of a stable solution of (\ref{hit2}) there exists a unique 
$\G$-orbit of solutions of (\ref{hit1}), (\ref{hit2}). 
The deformation theory of the equations at $(\bar{\partial}_A,\Phi)$ is then governed by the elliptic complex 
\begin{equation}\label{eq:l2compl}
	L^2 (C,\Omega^0\otimes \gl(E)) \xrightarrow{\d_1} L^2 (C,(\Omega^{1,0} \oplus \Omega^{0,1})\otimes \gl(E)) \xrightarrow{\d_2} L^2 (C,\Omega^2\otimes \gl(E)) 
\end{equation}
with 
\begin{equation}\label{eq:d1}
	\d_1(\dot{g}) = \bar{\partial}_A(\dot{g}) + [\Phi,\dot{g}]
\end{equation}
the infinitesimal action of $\G_{\C}$ 
(and where $\bar{\partial}_A$ acts on $\dot{g} \in L^2 (C,\Omega^0\otimes \gl(E))$ by the adjoint-action) and 
\begin{equation}\label{eq:d2}
	\d_2(\dot{A} , \dot{\Phi}) = \bar{\partial}_A(\dot{\Phi}) + [\dot{A},\Phi]
\end{equation}
for any 
\begin{equation}\label{eq:tangentvector}
	\dot{A} \in L^2 (C,\Omega^{0,1}\otimes \gl(E)) , \quad \dot{\Phi} \in L^2 (C,\Omega^{1,0}\otimes \gl(E)) ,
\end{equation}
and the action of $\bar{\partial}_A$ on $\dot{\Phi}$ is again induced by the adjoint-action. 
Alternatively, the cohomology groups of the elliptic complex (\ref{eq:l2compl}) compute the hypercohomology groups $\H^{\bullet}$ of 
the algebraic complex of sheaves on $C$ 
\begin{equation}\label{eq:Dolcompl}
		\End(\E) \xrightarrow{\ad_{\Phi}} \End(\E)\otimes K_C 
\end{equation}
where $\E = (E,\bar{\partial}_A)$ is the holomorphic vector bundle defined by the operator $\bar{\partial}_A$ and 
$\End(\E)$ its bundle of holomorphic endomorphisms. 
In this description, the Dolbeault complex structure $I$ and associated holomorphic symplectic form $\Omega_I$ on the 
moduli space are defined by 
\begin{align}
	I(\dot{A} , \dot{\Phi}) & =  (\sqrt{-1}\dot{A} , \sqrt{-1}\dot{\Phi}) \label{eq:Dol-hol-str}\\
	\Omega_I ((\dot{A} , \dot{\Phi}), (\dot{B} , \dot{\Psi})) & = \int_C \tr (\dot{\Psi}\wedge\dot{A} - \dot{\Phi}\wedge\dot{B} ).
	\label{eq:Dol-hol-sympl-str}
\end{align}

Hitchin's results were generalised to certain singular solutions by O. Biquard and P. Boalch \cite{Biq-Boa}. 
Namely, fix points $z_1,\ldots,z_n\in C$ and consider meromorphic Higgs fields $\Phi$ such that 
\begin{itemize} 
\item with respect to a holomorphic local coordinate $z$ centered at $z_j$ and 
some local $\bar{\partial}_A$-holomorphic frame of $\E$ the endomorphism $\Phi$ admits an expansion of the form 
\begin{equation}\label{eq:formaltype}
	\Phi = (A^j_{m_j} z^{-m_j-1} + \cdots + A^j_1 z^{-2} ) \d z
\end{equation}
(called its irregular type) 
up to terms of pole order at most $1$ where $A^j_m\in\gl_r(\C)$ are fixed diagonal matrices.
\end{itemize}
We let $H_j$ denote the centraliser of the irregular type (\ref{eq:formaltype}) in $\Gl_r(\C)$ 
(i.e. the group of matrices simultaneously commuting with all terms $A^j_m$ of the irregular type) 
and $\h_j\subset\gl_r(\C)$ stand for its Lie algebra. Fix also rational numbers 
$$
	0 \leq \tilde{\alpha}^j_1 \leq \cdots \leq \tilde{\alpha}^j_r <1
$$
called parabolic weights. The Hermitian metric $h$ near $z_j$ is required to be mutually bounded with 
\begin{equation}\label{eq:adaptedmetric}
	\diag (|z|^{2\tilde{\alpha}^j_l})_{l=1}^r
\end{equation}
with respect to a holomorphic local frame as above (i.e. in which $A^j_{m_j}$ are diagonal). 
Let us relabel the distinct parabolic weights by 
\begin{equation}\label{eq:parweights}
		0 \leq \alpha^j_1 < \cdots < \alpha^j_{r'(j)} <1
\end{equation}
for some $r'(j)\leq r$. The behaviour of $h$ can then be used to define a flag 
\begin{equation}\label{eq:parflag}
	E_{|z_j}  = F^j_0 \supseteq F^j_1 \supseteq \cdots \supseteq F^j_{r'(j)-1} \supseteq  F^j_{r'(j)} = \{ 0\}
\end{equation}
in the fiber of $E$ at $z_j$ by letting $F^j_{l-1}$ consist of the evaluations at $z_j$ of the local sections whose 
squared norm can be bounded from above by 
$$
	\mbox{constant}\times |z|^{2\alpha^j_l}. 
$$
In addition to the above assumptions on the irregular type, one also needs to assume that 
\begin{itemize}
\item the residue $\res_{z_j}(\Phi)$ belongs to $\h_j$ and preserves the filtration (\ref{eq:parflag}); in particular, it induces  
	endomorphisms $\gr_l^F\res_{z_j}(\Phi)\in \h_j \cap \gl(F_{l-1}^j/F_l^j)$ on the graded pieces $F_{l-1}^j/F_l^j$ of (\ref{eq:parflag}) 
\item the graded pieces $\gr_l^F\res_{z_j}(\Phi)$ of the residue of $\Phi$ at $z_j$ lie in prescribed adjoint orbits  
\begin{equation}\label{eq:adorbit}
	\O_l^j\subset\h_j
\end{equation} 
\end{itemize}
It is then shown in Theorem 5.4 of \cite{Biq-Boa} that if the irregular type (\ref{eq:formaltype}), adjoint orbit 
(\ref{eq:adorbit}) and parabolic weights $\alpha^j_l$ are fixed so that they satisfy the above conditions then the irreducible 
solutions of the system (\ref{hit1}), (\ref{hit2}) having these prescribed singular behaviours near the punctures 
(including (\ref{eq:adaptedmetric})) still form a smooth hyper-K\"ahler manifold 
\begin{equation}\label{eq:modulispace}
	\M 
\end{equation}
which in addition is complete if there are no reducible solutions and the adjoint orbits $\O^j_l$ of the residues are orbits of semi-simple matrices. 

The analytic deformation complex is a certain weighted Sobolev--space analog of (\ref{eq:l2compl}) and the tangent space of $\M$ at a solution 
$(\bar{\partial}_A,\Phi)$ is given by the first $L^2$-cohomology $L^2H^1(\gl(E))$ of the deformation complex for the Euclidean metric near 
the punctures. 
The associated Dolbeault holomorphic symplectic structure is then given by the same equations 
(\ref{eq:Dol-hol-str}), (\ref{eq:Dol-hol-sympl-str}) as in the non-singular case. 
The Dolbeault complex (\ref{eq:Dolcompl}) governing the algebraic deformations of the moduli space also 
needs to be modified accordingly as follows
\begin{equation}\label{eq:Dolcompl-irreg}
	\widetilde{\End}(\E) \xrightarrow{\ad_{\Phi}} \End_{iso}(\E)\otimes K_C
 \end{equation}
where 
\begin{itemize}
\item $\End_{iso}(\E)$ stands for the sheaf of meromorphic endomorphisms $\varphi$ of $\E$ with at most a simple pole near every $z_j$ 
	and with residue $\res_{z_j}(\varphi)\in\h_j$ respecting the parabolic flag (\ref{eq:parflag}) such that 
	each graded piece $\gr_l^F\res_{z_j}(\varphi)$ belongs to the image of the map $\ad_{\gr_l^F\res_{z_j}(\Phi)}$ 
	(called isomonodromic endomorphisms)
\item $\widetilde{\End}(\E)$ consists of the local holomorphic sections of $\End(\E)$ whose image by $\ad_{\Phi}$ belongs to 
	$\End_{iso}(\E)\otimes K_C$.
\end{itemize}
\begin{rem}
It would certainly be possible to describe $\widetilde{\End}(\E)$ explicitly as a certain iterated Hecke transform of $\End(\E)$ along a 
nested subsequence of Levi subgroups at the points $z_j$ corresponding to the successive centralisers of the matrices in the 
irregular type; however since such a description will be not needed for our purposes we do not give the details here. 
\end{rem}
\begin{prop}\label{prop:DolL2}
The deformation theory of the moduli space (\ref{eq:modulispace}) at a given solution $(A,\Phi)$ is governed 
by the hypercohomology groups of (\ref{eq:Dolcompl-irreg}).
\end{prop}
\begin{proof}
Let $g(t)\in\G_{\C}$ be a $1$-parameter family of global gauge transformations fixing $(A,\Phi)$ with tangent vector $\dot{g}$ 
an $L^2$-function with values in the bundle of endomorphisms of $E$. Then from (\ref{eq:d1}) one sees that 
$\dot{g}\in \ker(\bar{\partial}_A + \ad_{\Phi})$.  

Let  $(A(t),\Phi(t))$ be a $1$-parameter family of deformations parameterised by $t\in\C$ and consider the associated tangent vector 
$T = (\dot{A},\dot{\Phi})$ where $\dot{A}$ represents the deformation of the underlying parabolic vector bundle and 
$\dot{\Phi}$ that of the Higgs field as in (\ref{eq:tangentvector}). 
Then, as at any $z_j$ the Higgs fields $\Phi(t)$ have the same irregular type for all values of $t$, 
one sees that $\dot{\Phi}$ may have a pole of at most first order. 
Moreover, since $\Phi(t)$ preserve the parabolic flag, the same thing must also hold for $\dot{\Phi}$. 
Finally, as for any $l$ the adjoint orbit of the graded piece $\gr_l^F\res_{z_j}(\Phi(t))$ is constant in $t$, 
there exist local holomorphic gauge transformations $g(t)$ with $g(0)\in H_j$ and respecting the flag (\ref{eq:parflag}) 
so that the graded pieces of the residue of $g(t)\cdot \Phi(t)$ do no depend on $t$. 
By differentiation one finds that $\gr_l^F\res_{z_j}(\dot{\Phi})$ must lie in the tangent space of the adjoint orbit, 
i.e. in $\mbox{Im}(\ad_{\gr_l^F\res_{z_j}(\Phi)})$. 
\end{proof}

\section{Moduli spaces of stable sheaves of dimension $1$ on Hirzerbuch--surfaces}\label{sec:sheaves}

Let now the pair $(\bar{\partial}_A,\Phi) = (\E,\Phi)$ be a meromorphic solution of the equation (\ref{hit2}) as 
in Section \ref{sec:mod-Hit} and consider 
\begin{equation}\label{eq:divisor}
	D = (m_1+1) z_1 + \cdots + (m_n+1) z_n 
\end{equation}
as an effective divisor on $C$, $L = K_C(D)$ the line bundle of meromorphic $1$-forms with poles 
of order at most $m_l+1$ at $z_l$ for $l\in\{1,\ldots,n\}$ and $L^{\vee}$ its dual holomorphic line bundle. Let 
\begin{equation}\label{eq:Z}
	Z = \P(\O \oplus L) \xrightarrow{\pi} C
\end{equation}
be the associated Hirzerbuch--surface, i.e. the fibre-wise compactification of the total space $\L$ of $L$. 
Let $\O(1)$ denote the relatively ample line bundle on $Z$, $y\in H^0(Z,\O(1))$ the canonical section 
of $\O(1)$ and $x\in H^0(Z,\O(1)\otimes \pi^*L)$ the canonical section of $\O(1)\otimes \pi^*L$. 
Associate to $(\E,\Phi)$ the sheaf $M = M_{(\E,\Phi)}$ on $Z$ defined by the exact sequence 
\begin{equation}\label{eq:Mresolution}
	0\to \pi^*(\E \otimes L^{\vee})  \xrightarrow{\pi^*\Phi\otimes y - \pi^*{Id}_{\E}\otimes x} \pi^* \E\otimes \O(1) \to M_{(\E,\Phi)}\to 0. 
\end{equation}
Then \cite{BNR}, assuming the support $S_{(\E,\Phi)}$ of $M_{(\E,\Phi)}$ is reduced, 
$M_{(\E,\Phi)}$ is a torsion free sheaf of rank $1$ on $S_{(\E,\Phi)}$ and this association is an equivalence: 
the holomorphic bundle $\E$ is equal to the push-forward $\pi_*M$ and $\Phi$ is 
induced by the push-forward of the multiplication map by $x$ 
$$
	x : M \to M \otimes \pi_*L \otimes \O(1). 
$$
Throughout the paper we will denote the restriction of (\ref{eq:Z}) to $S_{(\E,\Phi)}$ by 
$$
	\pi_S:S_{(\E,\Phi)}\to C. 
$$


At this point we notice that $Z$ is a Poisson surface for the canonical Liouville symplectic form $\omega = \d \theta$ on $\L\subset Z$ 
with degeneracy divisor  
\begin{equation}\label{eq:degen-divisor}
	D_{\infty} = D + 2(y), 
\end{equation}
where 
$(y)$ is the divisor cut out by the section $y$ (divisor at infinity). 
The space of line bundles of a given degree $d$ supported on a deformation of a given smooth curve $S_0$ in a Poisson surface $X$ 
(the relative Picard bundle) has been studied by Donagi and Markman in \cite{DM}. It is shown in Theorem 8.3 {\em loc. cit.} that 
if $B$ stands for the Zariski open subset of the component containing $S_0$ of the Hilbert scheme parameterising smooth irreducible 
curves $S$ on $X$ not contained in $D_{\infty}$ then $B$ is smooth, moreover the relative Picard bundle 
\begin{equation}\label{eq:relPic}
	\Pic^d(X) \to B
\end{equation}
has a canonical Poisson structure whose symplectic leaves are obtained by prescribing the intersection of the curves $S$ with $D_{\infty}$. 
In the case of the surface (\ref{eq:Z}) with (\ref{eq:degen-divisor}) this means that we prescribe the intersection schemes of the curves $S$ with 
the fibers of $\pi$ over the points $z_i$ taken with multiplicity $(m_i+1)$, i.e. the jets of order $(m_i+1)$ of $S$. 
If $S$ is the spectral curve of a Higgs bundle as in the previous paragraph, this amounts to fixing the singularity behaviour 
of the Higgs field at the punctures. We also note by passing that the generic spectral 
curve in a given symplectic leaf might not be smooth in $Z$ precisely for this reason (e.g. $m_i=0$ and a non-regular semi-simple 
residue imply a node of $S$ in the fiber over $z_i$); however, the generic spectral curve is smooth in a suitable blow-up of $Z$ \cite{ASz}. 
For any $M\in\Pic^d$ with smooth underlying curve $S$ the deformation theory of $\Pic^d$ is governed by the global 
$\GlobExt$-groups $\GlobExt_{\O_Z}^{\bullet}(M,M)$ and the Poisson structure on (\ref{eq:relPic}) is a generalisation of Mukai's pairing 
$\Omega_{\Mukai}$ \cite{Muk} which is in turn defined by the Yoneda product 
\begin{equation}\label{eq:Mukaiform}
	\GlobExt_{\O_Z}^1(M,M) \times \GlobExt_{\O_Z}^1(M,M) \to \GlobExt_{\O_Z}^2(M,M)
\end{equation}
followed by Serre duality. 

We may rephrase the discussion of the previous two paragraphs by saying that there is a one-to-one correspondence between 
Zariski open subsets of the moduli space $\M$ of meromorphic Higgs bundles $(\E,\Phi)$ and a symplectic leaf of $\Pic^d(Z)$ 
for some value of $d$. 
The main idea in the proof of Theorem \ref{thm:main} will consist in matching the Dolbeault holomorphic symplectic structure of $\M$ 
with Mukai's symplectic structure on the symplectic leaves of $\Pic^d(Z)$.


\section{Fourier--Laplace--Nahm transform of connections on the Riemann sphere}\label{sec:FLN}

We take up the notations of Section \ref{sec:mod-Hit} with $C=\C P^1$ the complex projective line. 
The singularities will be fixed as follows: 
\begin{itemize}
\item regular singularities (i.e. $0$ irregular type) at $z_j$ with semi-simple residue in the adjoint orbit of 
$$
	\Lambda^j\in\gl_r(\C) = \diag (0,\ldots,0,\lambda^j_{r_j+1},\ldots,\lambda^j_r)
$$ 
	satisfying $\lambda^j_n \neq \lambda^j_m$ for $n> r_j$ and $m\neq n$
\item parabolic weights $\alpha^j_m$ such that $\alpha^j_m = 0 \iff \lambda^j_m = 0$
\item an irregular singularity with $m_{\infty}=1$,  
$$
	A^{\infty}_1 = \diag (\xi_1,\ldots,\xi_1,\ldots \ldots \xi_{\hat{n}},\ldots,\xi_{\hat{n}})
$$
with $\xi_l\neq \xi_m$ for $n\neq m$, the multiplicity of $\xi_l$ being $m_l$, and residue in the adjoint orbit of the element 
$$
	\Lambda^{\infty} = \diag (\lambda^{\infty,1}_1,\ldots ,\lambda^{\infty,1}_{m_1}) \oplus \cdots \oplus 
	\diag(\lambda^{\infty,\hat{n}}_1,\ldots ,\lambda^{\infty,\hat{n}}_{m_{\hat{n}}})
$$
of $\h = \gl_{m_1}(\C) \oplus \cdots \oplus \gl_{m_{\hat{n}}}(\C)$ such that in every eigenspace of $A^{\infty}_1$ 
(corresponding to a term in the decomposition of $\h$) the eigenvalues of $\Lambda^{\infty}$ are distinct
\item parabolic weights $\alpha^{\infty,l}_m>0$. 
\end{itemize}
In \cite{Sz-these} we studied Nahm transform $\Nahm$ for solutions of this type, showed that the category consisting of 
such objects is preserved under the transform, and described the transformation of the singularity parameters 
$z_i,\lambda^i_m,\alpha^j_m,\xi_l,\lambda^{\infty,l}_m,\alpha^{\infty,l}_m$ given above. 
Furthermore, we showed that the transform is an involution (up to a sign), in particular it is a diffeomorphism 
\begin{equation}\label{eq:Nahm}
	\Nahm : \M \to \Mt
\end{equation}
between moduli spaces of solutions of the Hitchin--equations with the given singularity behaviour. 
The transform is defined using a dual variable $\xi\in\Ct$ and twisting the Higgs bundle according to the 
formulas 
$$
	\bar{\partial}_{A}, \Phi -  \frac{\xi}2 \d z, 
$$
letting $\Et\to\Ct$ be the index bundle of the family (i.e. $\Et_{\xi}$ is the cokernel of the associated Dirac-operator), 
$\bar{\partial}_{\hat{A}}$ be induced by the trivial holomorphic bundle of the ambient Hilbert bundle, and 
$\widehat{\Phi}$ induced by multiplication by $-\frac z2\d\xi$. 
The extension over the points $\xi = \xi_l$ can be given either a metric definition (using that one has control on the 
$L^2$-norms of the $1$-forms in the kernel of the adjoint Dirac operator) or alternatively an algebraic one 
(using certain Hecke transforms of sheaves \cite{ASz}). We can now make precise the meaning of Theorem \ref{thm:main}. 

\begin{thm}
The transform (\ref{eq:Nahm}) is a hyper-K\"ahler isometry.
\end{thm}

\begin{proof}
According to Proposition \ref{prop:holomsymp-hypK}, it is sufficient to show that $\Nahm$ preserves 
the Dolbeault and de Rham complex structures and the Dolbeault holomorphic symplectic structure. 
The statement for the  Dolbeault complex structure is shown in \cite{ASz} where we give the birational 
transformations needed to define the transform on the level of spectral sheaves. 
Preservation of the de Rham complex structure is proved in \cite{Sz-FourierLaplace} where we identify 
the transformation of the meromorphic connection with the Fourier--Laplace transform of the underlying 
holonomic $\mathcal{D}$-module. Therefore, it is sufficient to prove the following. 
\begin{prop}\label{prop:NahmDolbeault}
The pull-back by $\Nahm$ of the Dolbeault holomorphic symplectic structure $\Omega_{\hat{I}}$ on $\Mt$ is $\Omega_I$ on $\M$.
\end{prop}
The rest of this section is devoted to proving this proposition. 
Consider two $1$-parameter families $(\E(t),\Phi(t)),(\E(x),\Phi(x))$ of elements of $\M$ for $t,x\in \C$ specialising to the same point 
$(\E,\Phi)\in\M$ at $t=0$ and $x=0$ respectively, giving rise to tangent vectors $T,X\in T_{(\E,\Phi)}\M$. 
On the other hand, consider the associated families of spectral sheaves $M_{(\E(t),\Phi(t))}$ and $M_{(\E(x),\Phi(x))}$ 
in $\Pic^d(Z)$; they give rise to tangent vectors $\tilde{T},\tilde{X}\in T_{M_{(\E,\Phi)}}\Pic^d(Z)$. The key fact is that 
\begin{equation}\label{eq:DolbeaultMukai}
	\Omega_I(T,X) = \Omega_{\Mukai}(\tilde{T},\tilde{X}). 
\end{equation}
This observation is due to J. Hurtubise \cite{Hur} in the case of regular Higgs bundles and J. Harnad and J. Hurtubise \cite{Har-Hur} 
in the meromorphic setting. The proof of \cite{Hur} relies on an explicit residue calculation, the one of \cite{Har-Hur} takes place 
in a more general setting of multi-Hamiltonian structures but is restricted to the case of a trivial bundle over $\C P^1$; 
in order to be self-contained, we give below an algebraic geometric proof valid for Higgs bundles with arbitrary underlying 
holomorphic bundle over any curve, with semi-simple irregular singularities as in Section \ref{sec:mod-Hit}. 

Let us show how (\ref{eq:DolbeaultMukai}) implies Proposition \ref{prop:NahmDolbeault}: 
it follows from \cite{Sz-these} and \cite{ASz} that $(\E,\Phi)$ and $\Nahm(\E,\Phi)$ have isomorphic spectral sheaves 
\begin{equation*}
	M_{(\E,\Phi)} \cong M_{(\widehat{\E},\widehat{\Phi})}.  
\end{equation*}
The same statement then clearly holds for families too 
\begin{equation}\label{eq:isomspec}
	M_{(\E(t),\Phi(t))} \cong M_{(\widehat{\E}(t),\widehat{\Phi}(t))}  
\end{equation}
(and similarly for $t$ replaced by $x$). Let now $\hat{T}=D_{(\E,\Phi)}\Nahm(T),\hat{X}=D_{(\E,\Phi)}\Nahm(X)$ \
be the tangent vectors defined by the families 
$$
	(\widehat{\E}(t),\widehat{\Phi}(t)), (\widehat{\E}(x),\widehat{\Phi}(x)); 
$$
then according to (\ref{eq:isomspec}) and (\ref{eq:DolbeaultMukai}) both 
$$
	\Omega_I(T,X) \quad \mbox{and}  \quad \Omega_{\hat{I}}(\hat{T},\hat{X})
$$
are equal to 
$$
	\Omega_{\Mukai}(\tilde{T},\tilde{X}). 
$$

Therefore, there only remains to prove the identity (\ref{eq:DolbeaultMukai}). 
For this purpose we first need to choose a complex whose hypercohomology groups computes  
$\GlobExt_{\O_Z}^{\bullet}(M_{(\E,\Phi)},M_{(\E,\Phi)})$. 
As we have a syzygy for $M_{(\E,\Phi)}$ it is natural to take the sheaves of homomorphisms of $\O_Z$-modules from (\ref{eq:Mresolution})  to 
$M_{(\E,\Phi)}$: 
\begin{equation}\label{eq:dualsyzygy}
	\Hom_{\O_Z}(\pi^* \E \otimes \O(1),M_{(\E,\Phi)})
	\longrightarrow
	\Hom_{\O_Z}(\pi^*(\E\otimes_{\O_C} L^{\vee}),M_{(\E,\Phi)})
\end{equation}
with the non-zero sheaves in degrees $0$ and $1$ and morphism 
\begin{equation}\label{eq:dualsyzygymorphism}
	\Hom(\pi^*\Phi\otimes y - \pi^*{Id}_{\E}\otimes x , Id_{M}).
\end{equation}
\begin{lem}\label{lem:Dolcompl-Poisson}
The push-forward $R\pi_*$ of (\ref{eq:dualsyzygy}) is 
\begin{equation}\label{eq:dualDolcompl-Poisson}
	\End_{\O_C}(\E)
	\xrightarrow{\ad_{\Phi}}
	\End_{\O_C}(\E)\otimes_{\O_C} L. 
\end{equation}
In particular, the tangent space $T_{M_{(\E,\Phi)}}\Pic^d(Z)$ is naturally isomorphic to the first hypercohomology 
$\H^1(\ad_{\Phi})$ of (\ref{eq:dualDolcompl-Poisson}). 
\end{lem}
\begin{proof}[Proof of the Lemma]
The support of the sheaves appearing in (\ref{eq:dualsyzygy}) is finite over $C$, so the action of $R\pi_!$ on the complex is 
equal to that of $R\pi_*$ which in turn is simply given by $\pi_*$. We have 
$$
		\Hom_{\O_Z}(\pi^* (\E\otimes_{\O_C} L^{\vee}) \otimes \O(1),M_{(\E,\Phi)})
		\cong 
		\Hom_{\O_Z}(\pi^* \E \otimes \O(1),M_{(\E,\Phi)} \otimes \pi^* L)
$$
and by the projection formula 
$$
	\pi_*(M_{(\E,\Phi)}\otimes \pi^* L) = \E\otimes_{\O_C} L.
$$ 
It then follows from local Verdier duality that the sheaves in the push-forward of (\ref{eq:dualsyzygy}) agree with the sheaves 
appearing in (\ref{eq:dualDolcompl-Poisson}).
As for the push-forward of the morphism, we know that multiplication by $x,y$ push down to $\Phi,\mbox{Id}_{\E}$ respectively 
on the second factor $\E$ in $\Hom(\E,\E)\otimes L = \Hom(\E,\E\otimes L)$. 
Hence using again the projection formula the push-forward of the map (\ref{eq:dualsyzygymorphism}) is 
$$
	\Phi \circ - \circ \Phi = \ad_{\Phi}. 
$$
\end{proof}

\begin{rem}\label{rem:symplectic-leaf}
We see that (\ref{eq:dualDolcompl-Poisson}) is not quasi-isomorphic to (\ref{eq:Dolcompl-irreg}) in agreement  
with our observation in Proposition \ref{prop:DolL2} that (\ref{eq:Dolcompl-irreg}) is the tangent space to 
a symplectic leaf of the larger Poisson space $\Pic^d(Z)$ where the irregular types and residues may vary; 
in fact (\ref{eq:Dolcompl-irreg}) is a subcomplex of (\ref{eq:dualDolcompl-Poisson}). 
\end{rem}

Let us therefore define the sub-complex 
\begin{equation}\label{eq:dualsyzygyiso}
		\widetilde{\Hom}_{\O_Z}(\pi^* \E \otimes \O(1),M_{(\E,\Phi)}) \longrightarrow 
			\Hom_{\O_Z,iso}(\pi^*(\E\otimes_{\O_C} L^{\vee}),M_{(\E,\Phi)})
\end{equation}
of (\ref{eq:dualsyzygy}) as the pull-back by $\pi$ of (\ref{eq:Dolcompl-irreg}). 
In concrete terms, a local section of 
$$
	\Hom_{\O_Z}(\pi^*(\E\otimes_{\O_C} L^{\vee}),M_{(\E,\Phi)})
$$ 
is defined to belong to $\Hom_{\O_Z,iso}$ if and only if its image by $\pi$ lies in $\End_{iso}(\E)\otimes K_C$, 
and similarly for $\widetilde{\Hom}_{\O_Z}$ and $\widetilde{\End}$. 
This subcomplex then governs the deformations of the sheaves giving rise to Higgs bundles with prescribed singularities 
near the punctures as in Section \ref{sec:mod-Hit}. 

We will use the Dolbeault resolution of the complex (\ref{eq:dualsyzygyiso}) to compute the groups 
$\GlobExt_{\O_Z}^1(M_{(\E,\Phi)},M_{(\E,\Phi)})$. 
Specifically, given tangent vectors 
$$
	T=[(\dot{A},\dot{\Phi})],X=[(\dot{B},\dot{\Psi})]\in \H^1(\ad_{\Phi}:\widetilde{\End}(\E)\to \End_{par}\otimes K_C) \cong T_{(\E,\Phi)}\M
$$ 
as in (\ref{eq:tangentvector}) we will construct first hypercohomology classes 
$$
	\tilde{T} = [(\tilde{\dot{A}},\tilde{\dot{\Phi}})],\tilde{X} = [(\tilde{\dot{B}},\tilde{\dot{\Psi}})]
$$ 
in the complex (\ref{eq:dualsyzygyiso}) mapping to $(\dot{A},\dot{\Phi})$ and $(\dot{B},\dot{\Psi})$ respectively  via the 
push-forward map of the lemma and check that the pairing (\ref{eq:Mukaiform}) applied to these lifts 
agrees with (\ref{eq:Dol-hol-sympl-str}). 
To find the lifts, we will use the $L^2$ Dolbeault resolution of the $\Hom$-sheaves in these complexes. 

Notice first that is is sufficient to prove (\ref{eq:DolbeaultMukai}) Zariski locally on $\M$, so we may 
assume that the schematic support $S=S_{(\E,\Phi)}$ of $M_{(\E,\Phi)}$ is reduced and smooth away from $\pi^{-1}(D)$. 
For such a Higgs bundle $(\E,\Phi)\in\M$, it is possible to find Zariski locally on $C$ a trivialisation 
$\e_1,\ldots,\e_r$ of $\E$ with respect to which $\Phi$ is diagonal: indeed, over the Zariski open 
set $U\subset C$ over which $S$ is non-ramified, the geometric fibers 
of $S$ consist of $r$ distinct points $x_1(z),\ldots,x_r(z)\in L|_z$ depending on $z\in U$ 
(where as usual we identify a point $[X:Y]\in Z$ with $x = X/Y \in L$ if $Y\neq 0$). 
Notice that these points might get permuted  as the point $z$ moves along a non-trivial loop in $U$; 
let us therefore restrict $z$ to lie in an analytic disc $\Delta \subset U$ where the labelling of 
$x_1(z),\ldots,x_r(z)$ may be well defined without any ambiguity. 
Then $S\cap\pi^{-1}(\Delta)$ consists of $r$ disjoint connected components $S_1,\ldots,S_r$ 
and the map $\pi_i = \pi|_{S_i} :S_i\to \Delta$ is biholomorphic. A holomorphic coordinate on $S_i$ 
can be chosen as $w_i = \pi_i^* z$. Furthermore, the restriction $M|_{S_i}$ of $M$ to $S_i$ is a torsion free sheaf, 
i.e. a line bundle since $S_i$ is a smooth curve. As $S_i$ is affine, $M|_{S_i}$ is trivial, and we let 
$\m_i$ stand for a generator. We now set $\e_i = \pi_*\m_i$; it is then clear that $\e_1,\ldots,\e_r$ 
trivialises $\E$ over $\Delta$ and that with respect to this trivialisation $\Phi$ is diagonal 
\begin{equation}\label{eq:Phidiag}
	\Phi(z) = \diag (x_1(z),\ldots,x_r(z)). 
\end{equation}
In particular, we infer that in the exact sequence (\ref{eq:Mresolution}) $\m_i$ is the class $[\e_i]$ of $\e_i$ modulo the image of $\Phi - x_i$. 
Now, the sheaves appearing in (\ref{eq:dualsyzygyiso}) are supported on $S_{(\E,\Phi)}$, and 
their restrictions to the local sheet $S_j$ are free $\O_{S_j}$-modules generated by the sections 
$$
	\pi^*\e_1^{\vee} \otimes \m_j,\ldots ,\pi^*\e_r^{\vee} \otimes \m_j 
$$
and 
$$
	\pi^*\e_1^{\vee} \otimes \m_j \d w_j,\ldots ,\pi^*\e_r^{\vee} \otimes \m_j \d w_j
$$
respectively, where $\otimes$ stands for $\otimes_{\O_{S_j}}$ and $\e_1^{\vee} ,\ldots ,\e_r^{\vee}$ denotes 
the dual basis of $\e_1,\ldots,\e_r$. Let us now use the trivialisation $\e_1,\ldots,\e_r$ to write 
\begin{align*}
	\dot{A} &= a \d \bar{z} & \dot{\Phi} & = \phi \d z \\
	\dot{B} &= b \d \bar{z} & \dot{\Psi} & = \psi \d z 
\end{align*}
where $a= (a_{ij}),b,\phi,\psi:\Delta\to\gl_r(\C)$ are $L^2$ functions satisfying 
\begin{equation}\label{eq:cocycleC}
	\bar{\partial}(\phi \d z) + [a \d \bar{z},\Phi] = 0, \quad \bar{\partial}(\psi \d z) + [b \d \bar{z},\Phi] = 0. 
\end{equation}
By Theorem 5.4 \cite{Biq-Boa}, the tangent space $T_{(\E,\Phi)}\M$ is isomorphic to the first $L^2$-cohomology of (\ref{eq:l2compl}). 
We define a Hermitian metric on $Hom(\pi_S^*E,M_{(\E,\Phi)})$ by letting 
$$
	\langle \pi^*\e_i^{\vee}\otimes \m_j , \pi^*\e_{i'}^{\vee}\otimes \m_{j'} \rangle 
		= \delta_{jj'} h(\e_j,\e_j) h (\e_i^{\vee}, \e_{i'}^{\vee} ) 
$$
where by an abuse of notation the metric $h$ on the right hand side refers both to the Hermitian metric on $E$ 
and the one induced by $h$ on $E^{\vee}$. 
We also define the Riemannian metric $|\d w_j|^2$ on $S_j$. 
For ease of notation we set $M = M_{(\E,\Phi)}$. 
Splitting $1$-forms according to their type it then follows that the Dolbeault resolution of the complex (\ref{eq:dualsyzygyiso}) is 
quasi-isomorphic to 
\begin{align}\label{diagram:Dolres}
	\xymatrix{
	Hom_{L^2(S)}(\pi_S^*E,M\otimes \Omega_S^{0,1}) \ar[r]^{\pi_S^*\Phi - x Id_E} & Hom_{L^2(S)}(\pi_S^*E,M\otimes \Omega_S^{1,1})  \\
	Hom_{L^2(S)}(\pi_S^*E,M) \ar[r]^{\pi_S^*\Phi - x Id_E} \ar[u]_{\bar{\partial}} & Hom_{L^2(S)}(\pi_S^*E,M\otimes \Omega_S^{1,0})
		 \ar[u]_{\bar{\partial}}
	}
\end{align}
where $\pi_S$ is the restriction of $\pi$ to $S=S_{(\E,\Phi)}$, the vertical maps are induced by 
$\pi_S^*\bar{\partial}^{\E}$ and the Dolbeault operator $\bar{\partial}^{M}$ of $M$. Define the elements 
\begin{align*}
	(\tilde{\dot{A}},\tilde{\dot{\Phi}}), (\tilde{\dot{B}},\tilde{\dot{\Psi}})\in &
	Hom_{L^2(\pi^{-1}(\Delta))} (\pi_S^*E,M_{(\E,\Phi)}) \otimes \Omega_S^{0,1} \\
	& \oplus Hom_{L^2(\pi^{-1}(\Delta))}(\pi_S^*E,M_{(\E,\Phi)})\otimes \Omega_S^{1,0}
\end{align*}
componentwise on $S_j$ by $(\tilde{a}_j\d \bar{w}_j,\tilde{\phi}_j\d w_j)$ and $(\tilde{b}_j\d \bar{w}_j,\tilde{\psi}_j\d w_j)$ respectively, 
where 
\begin{align}
	\tilde{a}_j (w_j): \pi_j^*\e_i & \mapsto a_{ij}(\pi(w_j)) \m_j \label{eq:A}\\
	\tilde{\phi}_j (w_j): \pi_j^*\e_i & \mapsto \phi_{ij}(\pi(w_j)) \m_j. \label{eq:Phi}
\end{align}
and 
\begin{align}
	\tilde{b}_j (w_j): \pi_j^*\e_i & \mapsto b_{ij}(\pi(w_j)) \m_j \label{eq:B}\\
	\tilde{\psi}_j (w_j): \pi_j^*\e_i & \mapsto \psi_{ij}(\pi(w_j)) \m_j. \label{eq:Psi}
\end{align}
\begin{lem}\label{lem:push-forward}
\begin{enumerate}
\item The above local definitions match up to define global $L^2$ sections 
$$
	(\tilde{\dot{A}},\tilde{\dot{\Phi}}),(\tilde{\dot{B}},\tilde{\dot{\Psi}}) 
$$
away from $\pi^{-1}(D_{\red}\cup R)$ where $D_{\red}$ is the reduced divisor of $D$ and $R$ is the branch locus of $\pi:S\to C$. 
\label{lem:push-forward1}
\item The push-forwards of these sections by $\pi_*$ are equal to $(\dot{A},\dot{\Phi})$ and $(\dot{B},\dot{\Psi})$, respectively. 
	\label{lem:push-forward2}
\item Furthermore, $(\tilde{\dot{A}},\tilde{\dot{\Phi}})$ and $(\tilde{\dot{B}},\tilde{\dot{\Psi}})$ are $1$-cocycles in the Dolbeault resolution (\ref{diagram:Dolres}). \label{lem:push-forward3}
\end{enumerate}
\end{lem}
\begin{proof}[Proof of the Lemma]
(\ref{lem:push-forward1}) As $z$ describes a loop about a branching point of $S$, the sheets $S_i$ undergo 
some permutation $\sigma\in\mathfrak{S}_r$. The trivialisations $\e_1,\ldots,\e_r$ and $\m_1,\ldots,\m_r$ correspondingly 
transform by applying $\sigma$ on their labels. Clearly, the maps (\ref{eq:A}) then do not change, just the order in which they are listed. 
The same argument works for the support of $D$. The section so obtained is $L^2$ because by the definition of the Hermitian 
metrics on $Hom$-sheaves $\pi$ preserves the Hermitian metrics. 

(\ref{lem:push-forward2}) It is enough to show this locally. 
By definition one has $\pi_*\m_i = \e_i$ and $\pi_*\d w_i = \d z$. 
To compute the action of the push-forward of $\tilde{\dot{\Phi}}$ on $\e_i$ we need to apply $\tilde{\dot{\Phi}}$ to $\pi^*\e_i$ 
and take the push-forward of the result. In view of (\ref{eq:Phi}) this is then given by 
$$
	\e_i \mapsto \sum_{j=1}^r \phi_{ij}(\pi(w_j)) \pi_*(\m_j \d w_j) = \sum_{j=1}^r \phi_{ij}(z) \e_j \d z
$$
which is precisely the action of $\dot{\Phi}$. This proves the first statement for $\tilde{\dot{\Phi}}$. 
The proof for the components $\dot{A},\dot{B},\dot{\Psi}$ is analogous. 

(\ref{lem:push-forward3}) 
Since the sections $\pi_j^*\e_i$ and $\m_j$ are by definition holomorphic for the 
complex structures $\pi_S^*\bar{\partial}^{\E}$ and $\bar{\partial}^{M}$ respectively, by (\ref{eq:Phi}) the matrix of 
$\bar{\partial}\tilde{\dot{\Phi}}$ in these trivialisations is simply 
$$
	\frac{\partial \phi_{ij}}{\partial \bar{w}_j} \d w_j \wedge \d \bar{w}_j, 
$$
i.e. the $i,j$-entry of $\pi_j^*(\bar{\partial}(\phi \d z))$. 
On the other hand, it follows from (\ref{eq:Phidiag}) that 
$$
	\tilde{\dot{A}}\circ\pi^*\Phi : \pi^*\e_i  \mapsto a_{ij} \m_j  x_i  \wedge \d \bar{w}_j 
$$
and by definition $x = x_j$ on $S_j$, hence 
$$
	(\pi_j^*\Phi - x Id_E ) (\pi_j^*\e_i) =a_{ij} \m_j   (x_i - x_j)  \wedge \d \bar{w}_j. 
$$
Now by (\ref{eq:Phidiag}) this expression is equal to the pull-back by $\pi_j$ of the $i,j$-entry of $[a \d \bar{z},\Phi]$ 
with respect to the trivialisation $\e_1,\ldots, \e_r$. Thus, the cocycle condition (\ref{eq:cocycleC}) for $(\dot{A},\dot{\Phi})$ 
implies the cocycle condition for $(\tilde{\dot{A}},\tilde{\dot{\Phi}})$.
\end{proof}
As (\ref{eq:dualsyzygyiso}) is a subcomplex of (\ref{eq:dualsyzygy}) whose hypercohomology groups $\H^{\bullet}$ are isomorphic to 
$\GlobExt_{\O_Z}^{\bullet}(M_{(\E,\Phi)},M_{(\E,\Phi)})$, it follows from the lemma that the couple 
$$
	(\tilde{\dot{A}},\tilde{\dot{\Phi}})
$$
represents an element of $\GlobExt_{\O_Z}^{1}(M_{(\E,\Phi)},M_{(\E,\Phi)}) = T_{M_{(\E,\Phi)}}\Pic^d$ 
mapping to the class of $(\dot{A},\dot{\Phi})$ in the hypercohomology group $\H^1(\ad_{\Phi})=T_{(\E,\Phi)}\M$ 
of the complex (\ref{eq:dualDolcompl-Poisson}).
The Yoneda product of the extension classes 
\begin{equation}\label{eq:liftedtangentvectors}
	\tilde{T} = [(\tilde{\dot{A}},\tilde{\dot{\Phi}})], \tilde{X} =[(\tilde{\dot{B}},\tilde{\dot{\Psi}})] \in\GlobExt_{\O_Z}^1(M_{(\E,\Phi)},M_{(\E,\Phi)})
\end{equation}
is then a class in  $\GlobExt_{\O_Z}^2(M_{(\E,\Phi)},M_{(\E,\Phi)})$. In terms of the Dolbeault resolution (\ref{diagram:Dolres}) 
of (\ref{eq:dualsyzygyiso}), this latter group has a subgroup represented by sections of $Hom_{L^2(S)}(\pi_S^*E,M)\otimes \Omega_S^{2}$ 
to which the Yoneda product belongs. 
The paring is defined using composition of homomorphisms coupled with exterior product of differential forms. 
As $\Omega_S^{2,0} = 0 = \Omega_S^{0,2}$ the pairing is trivial between $\tilde{\dot{A}}$ and $\tilde{\dot{B}}$, 
and similarly for $\tilde{\dot{\Phi}}$ and $\tilde{\dot{\Psi}}$. 
On an analytic disc $\Delta$ away from the branch locus $R$ of $S\to C$ where the formulae (\ref{eq:A})--(\ref{eq:Psi}) hold, 
the homomorphism part of the pairing of $\tilde{\dot{A}}$ and $\tilde{\dot{\Psi}}$ is given as follows: 
on the right hand side of (\ref{eq:A}) we take representatives $\e_j$ for $\m_j$, then we apply (\ref{eq:Psi}) to these representatives 
\begin{align}
	\pi^* \e_i  & \mapsto (a_{i1}\tilde{\psi}(\pi^*\e_1), \ldots, a_{ir}\tilde{\psi}(\pi^*\e_r)) \notag \\
			& = (\sum_j a_{ij}\psi_{j1}\m_1, \ldots, \sum_j a_{ij}\psi_{jr}\m_r). \label{eq:acomposepsi}
\end{align}

\begin{lem}
The Poisson bivector induces an isomorphism  
\begin{equation}\label{eq:ext2h1}
	\GlobExt_{\O_Z}^2(M_{(\E,\Phi)},M_{(\E,\Phi)}) \cong H^1(S_{(\E,\Phi)},K_S(-(S\cap \pi^{-1}(D))))
\end{equation}
(where $S = S_{(\E,\Phi)}$). 
\end{lem}
\begin{proof}[Proof of the Lemma]
The argument follows Remark 8.17 of \cite{DM}; see also Section 5.4 of \cite{GH}. 
The standard spectral sequence (associated to the filtration by the degree of a projective resolution) for the groups 
$\GlobExt_{\O_Z}^{\bullet}(M_{(\E,\Phi)},M_{(\E,\Phi)})$ has 
$$
	E_2^{p,q} = H^p(Z, \Ext_{\O_Z}^q(M_{(\E,\Phi)},M_{(\E,\Phi)})), \quad E_{\infty}^{p,q} \Rightarrow 
	\GlobExt_{\O_Z}^{p+q}(M_{(\E,\Phi)},M_{(\E,\Phi)}).
$$
Now the sheaves $\Ext_{\O_Z}^q(M_{(\E,\Phi)},M_{(\E,\Phi)})$ are supported on the smooth curve $S_{(\E,\Phi)}$, so 
$E_2^{p,q} = 0$ for all $p\geq 2$. On the other hand, since these sheaves 
can be computed using a length $2$ projective resolution of $M_{(\E,\Phi)}$ we also have 
$E_2^{p,q} = 0$ for all $q\geq 2$. It follows that the only term of total degree $2$ in $E_2^{p,q}$ is 
$H^1(Z, \Ext_{\O_Z}^1(M_{(\E,\Phi)},M_{(\E,\Phi)}))$. It is also clear from the above vanishing properties that 
the differential $d_2$ vanishes, so that the spectral sequence degenerates at the second term $E_2$: 
$$
	\GlobExt_{\O_Z}^{2}(M_{(\E,\Phi)},M_{(\E,\Phi)}) = H^1(Z, \Ext_{\O_Z}^1(M_{(\E,\Phi)},M_{(\E,\Phi)})).
$$
Finally, notice that a local trivialisation $\m$ of $M_{(\E,\Phi)}$ on $S_{(\E,\Phi)}$ in the analytic topology 
induces the local trivialisation $\m^{\vee}\otimes\m$ of $\Hom(M_{(\E,\Phi)},M_{(\E,\Phi)})$ which in turn yields an isomorphism  
$$
	\Ext_{\O_Z}^1(M_{(\E,\Phi)},M_{(\E,\Phi)})\cong \Ext_{\O_Z}^1(\O_S,\O_S) \cong N_{S|Z}
$$
from the first $\Ext$-sheaf to the normal bundle of $S=S_{(\E,\Phi)}$ in $Z$. The Poisson bivector then identifies 
$$
	N_{S|Z} \cong K_S(-(S\cap \pi^{-1}(D)))
$$
because the curve $S \setminus D_{\infty}$ is Lagrangian in the open symplectic surface $Z \setminus D_{\infty}$ (where 
$D_{\infty} = D + 2(y)$ is the degeneracy divisor of the Poisson bivector) and $S$ does not intersect the divisor at infinity $(y)$. 
\end{proof}

It follows from the proof of the lemma that the identification between elements of 
$$
	Hom_{L^2(S)}(\pi_S^*E,M)\otimes \Omega_S^{2}
$$
and cohomology classes in (\ref{eq:ext2h1}) is induced in the local description (\ref{eq:acomposepsi}) by taking 
the lift $\e_i$ of $\m_i$ and associating to a given homomorphism as above the coefficient of 
$\pi^* \e_i^{\vee}\otimes\m_i$ on the sheet $S_i$. 
Therefore, the Yoneda product of the classes $\tilde{\dot{A}}, \tilde{\dot{\Psi}}$ is locally represented by the $(1,1)$-form 
\begin{equation*}
	\tr (\tilde{\dot{\Psi}}\wedge \tilde{\dot{A}}) = \sum_{j=1}^r a_{ij}\psi_{ji} \d w_i \wedge \d \bar{w}_i 
\end{equation*}
on the sheet $S_i$ for all $1\leq i\leq r$. A similar analysis goes through for the couple $\tilde{\dot{B}},\tilde{\dot{\Phi}}$, 
and the Yoneda product of the classes in (\ref{eq:liftedtangentvectors}) is represented by 
\begin{equation}\label{eq:trace}
	\sum_{j=1}^r (a_{ij}\psi_{ji} - b_{ij}\phi_{ji}) \d w_i \wedge \d \bar{w}_i 
\end{equation}
on the sheet $S_i$. 

Finally, we need to apply Serre duality to this class. Notice that the Serre dual of the vector space on the right hand side of 
(\ref{eq:ext2h1}) is 
$$
	H^0(S,\O_S(S\cap \pi^{-1}(D))). 
$$
It fits into the short exact sequence 
$$
	0 \to H^0(S,\O_S) \to H^0(S,\O_S(S\cap \pi^{-1}(D))) \to \oplus_{k=1}^n\oplus_{i=1}^r \C_{(z_k,x_i)}^{m_k+1} \to 0 
$$
where $x_1,\ldots, x_r\in L_{z_k}$ correspond to the intersection points $S\cap\pi^{-1}(z_k)$. 
This exact sequence splits and the map 
$$
	H^1(S_{(\E,\Phi)},K_S(-(S\cap \pi^{-1}(D)))) \times \oplus_{k=1}^n \oplus_{i=1}^r \C_{(z_k,x_i)}^{m_k+1} \to \C
$$
describes the infinitesimal variation of the singularity behaviour of the Higgs field associated to a variation of spectral sheaf. 
As in our case we assumed that the classes $\dot{A},\dot{\Psi}$ do not modify the polar part of the Higgs field, it follows that 
the above map is $0$. On the other hand, the map 
$$
	H^1(S_{(\E,\Phi)},K_S(-(S\cap \pi^{-1}(D)))) \times H^0(S,\O_S) \to \C
$$
is simply given by integration of smooth two-forms on $S$. For any $\varepsilon > 0$ let  
$$
	U_{\varepsilon} \subset C 
$$
be the complement of the union of the closed discs $\bar{B}_{\varepsilon}$ of radius $\varepsilon$ centered at the points of $D_{\red}\cup R$. 
Consider a finite covering 
$$
	U_{\varepsilon}  \subset \Delta_1 \cup \cdots \cup \Delta_N
$$
(where $N$ might depend on $\varepsilon$) by analytic discs as above and 
$$
	\rho_1, \ldots, \rho_N
$$
a partition of unity subordinate to this covering. They induce a covering 
$$
	\pi_S^{-1}(\Delta_{\alpha})
$$
of $\pi_S^{-1}(U_{\varepsilon})$ and a partition of unity 
$$
	\pi_S^*(\rho_1) , \ldots, \pi_S^*(\rho_N)
$$ 
subordinate to it. 
We find 
\begin{align*}
	\Omega_{\Mukai} (\tilde{T},\tilde{X})  & = 
		\lim_{\varepsilon\to 0}\int_{\pi_S^{-1}(U_{\varepsilon})} \sum_{\alpha =1}^N \pi_S^*(\rho_{\alpha}) 
			\tr(\tilde{\dot{\Psi}}\wedge \tilde{\dot{A}} - \tilde{\dot{\Phi}}\wedge \tilde{\dot{B}} ) \\
		& = \lim_{\varepsilon\to 0} \sum_{\alpha =1}^N 
		\sum_i \int_{S_{i,\alpha}}\pi_{i,\alpha}^*(\rho_{\alpha}) \sum_j (a_{ij}\psi_{ji} - b_{ij}\phi_{ji}) \d w_{i,\alpha} \wedge \d \bar{w}_{i,\alpha}
\end{align*}
where the second equality follows from the formula (\ref{eq:trace}), and the labeling of the sheets $S_{1,\alpha},\ldots,S_{r,\alpha}$ 
of $\pi_S^{-1}(\Delta_{\alpha})$ makes sense locally over $\Delta_{\alpha}$. Now the changes of variables 
$$
	w_{i,\alpha}\in S_{i,\alpha} \mapsto z_{\alpha}  \in \Delta_{\alpha} 
$$
transform the above expression into 
\begin{align}
	& \lim_{\varepsilon\to 0} \sum_{\alpha =1}^N \int_{\Delta_{\alpha}} \rho_{\alpha} 
	\sum_i \sum_j (a_{ij}\psi_{ji} - b_{ij}\phi_{ji})  \d z_{\alpha}  \wedge \d\bar{z}_{\alpha}  \notag \\
	& = \lim_{\varepsilon\to 0} \int_{U_{\varepsilon}} \tr (\dot{\Psi}\wedge \dot{A} - \dot{\Phi}\wedge \dot{B}) . \label{eq:limitintegral}
\end{align}
Now near a puncture $z_k$ in general there does not exist a trivialisation $\e_1,\ldots, \e_r$ in which $\Phi$ is diagonal. 
However, by Lemma 5.3 \cite{Biq-Boa} there exist $C,\delta > 0$ and a local trivialisation $\e_1,\ldots ,\e_r$ such that 
for all $i,j$ the harmonic representatives $a,\psi$ satisfy the estimates 
$$
	|a_{ij}| \leq C |z|^{-1+\delta} \quad |\psi_{ji}| \leq C |z|^{-1+\delta}
$$
where $z$ is a local coordinate centered at $z_k$. We infer 
$$
	|a_{ij}\psi_{ji}| \leq  C^2 |z|^{-2+2\delta}
$$
and similarly for $b\phi$. As a consequence one gets 
$$
	\int_{\bar{B}_{\varepsilon}(z_k)} |a_{ij}\psi_{ji}| \d z \wedge \d \bar{z} \leq  C_2 \varepsilon^{2\delta} \to 0
$$
as $\varepsilon\to 0$. 
On the other hand, near a point $r_k$ of the branch locus $R$ it follows from usual elliptic regularity that 
the harmonic representatives $a,\psi$ are smooth, so in any local trivialisation $\e_1,\ldots ,\e_r$ clearly 
$$
	\int_{\bar{B}_{\varepsilon}(r_k)} |a_{ij}\psi_{ji}| \d z \wedge \d \bar{z} \to 0
$$
as $\varepsilon\to 0$. 
Together these estimates imply that the limit (\ref{eq:limitintegral}) coincides with (\ref{eq:Dol-hol-sympl-str}).
\end{proof}

\bibliography{isometry}

\end{document}